\pdfoutput=1
\RequirePackage{ifpdf}
\ifpdf 
\documentclass[pdftex]{sigma}
\else
\documentclass{sigma}
\fi

\numberwithin{equation}{section}

\newtheorem{Theorem}{Theorem}[section]
\newtheorem{Corollary}[Theorem]{Corollary}
\newtheorem{Lemma}[Theorem]{Lemma}
\newtheorem{Proposition}[Theorem]{Proposition}
 { \theoremstyle{definition}
\newtheorem*{rems}{Remarks}
\newtheorem*{rem}{Remark} }

\newcommand{\rfac}[2]{{\left({#1}\right)_{#2}}}
\newcommand{\qrfac}[2]{{\left({#1}\right)_{#2}}}
\newcommand{\pqrfac}[3]{{\left({#1};#3\right)_{#2}}}
\newcommand{\elliptictheta}[1]{\theta\!\left({#1} \right) }
\newcommand{\ellipticthetap}[1]{\theta\!\left({#1} ; p\right) }
\newcommand{\ellipticqrfac}[2]{{\left({#1}; q, p\right)_{#2}}}

\begin{document}

\allowdisplaybreaks

\newcommand{\arXivNumber}{1802.09885}

\renewcommand{\thefootnote}{}

\renewcommand{\PaperNumber}{052}

\FirstPageHeading

\ShortArticleName{The Determinant of an Elliptic Sylvesteresque Matrix}

\ArticleName{The Determinant of an Elliptic Sylvesteresque Matrix\footnote{This paper is a~contribution to the Special Issue on Elliptic Hypergeometric Functions and Their Applications. The full collection is available at \href{https://www.emis.de/journals/SIGMA/EHF2017.html}{https://www.emis.de/journals/SIGMA/EHF2017.html}}}

\Author{Gaurav BHATNAGAR and Christian KRATTENTHALER}

\AuthorNameForHeading{G.~Bhatnagar and C. Krattenthaler}

\Address{Fakult\"at f\"ur Mathematik, Universit\"at Wien, Oskar-Morgenstern-Platz 1, 1090 Wien, Austria}
\Email{\href{mailto:bhatnagarg@gmail.com}{bhatnagarg@gmail.com}, \href{mailto:christian.krattenthaler@univie.ac.at}{christian.krattenthaler@univie.ac.at}}
\URLaddress{\url{http://www.gbhatnagar.com}, \url{http://www.mat.univie.ac.at/~kratt/}}

\ArticleDates{Received February 28, 2018, in final form May 27, 2018; Published online May 30, 2018}

\Abstract{We evaluate the determinant of a matrix whose entries are elliptic hypergeometric terms and whose form is reminiscent of Sylvester matrices. A hypergeometric determinant evaluation of a matrix of this type has appeared in the context of approximation theory, in the work of Feng, Krattenthaler and Xu. Our determinant evaluation is an elliptic extension of their evaluation, which has two additional parameters (in addition to the base $q$ and nome $p$ found in elliptic hypergeometric terms). We also extend the evaluation to a formula transforming an elliptic determinant into a multiple of another elliptic determinant. This transformation has two further parameters. The proofs of the determinant evaluation and the transformation formula require an elliptic determinant lemma due to Warnaar, and the application of two $C_n$ elliptic formulas that extend Frenkel and Turaev's $_{10}V_9$ summation formula and $_{12}V_{11}$ transformation formula, results due to Warnaar, Rosengren, Rains, and Coskun and Gustafson.}


\Classification{33D67; 15A15}

\renewcommand{\thefootnote}{\arabic{footnote}}
\setcounter{footnote}{0}

\section{Introduction}
The determinant of a Sylvester matrix is used to determine whether two polynomials have a~common root. Recently, in the context of approximation theory, a determinant of a hypergeometric matrix was evaluated by Feng, Xu and the second author \cite{FKX2017}. The matrix they considered resembles a Sylvester matrix. The objective of this paper is to give an elliptic extension of their determinant evaluation.

We briefly discuss Sylvester matrices. Consider the Sylvester matrix that corresponds to the polynomials
\begin{gather*}
x^2+2x+1 = (x+1)^2
\end{gather*}
and
\begin{gather*}
x^3+3x^2+3x+1 = (x+1)^3
\end{gather*}
given by
\begin{gather*}
\begin{pmatrix}
1 & 2 & 1 & 0 & 0\\
0 & 1& 2 & 1 & 0\\
0 & 0 &1 & 2 & 1 \\
1 & 3 & 3 & 1 & 0\\
0& 1 & 3 & 3 & 1
\end{pmatrix}.
\end{gather*}
Its determinant is $0$, which indicates the obvious fact that the two polynomials $(x+1)^2$ and $(x+1)^3$ have a common root.

More generally, consider the Sylvester matrix defined as follows. Let
\begin{gather*}
b_{ij}(s, r) := \binom{r}{j-i} s^{j-i}
\end{gather*}
and consider the $ (r_1+r_2) \times (r_1+r_2)$ matrix $B=(b^{\prime}_{ij})$, where
\begin{gather*}
b^{\prime}_{ij} =\begin{cases}
b_{ij}(s_1,r_1), & \text{for } 0\leq i \leq r_2-1,\\
b_{i-r_2,j}(s_2,r_2), & \text{for } r_2\leq i \leq r_1+r_2-1 .
\end{cases}
\end{gather*}
Then its determinant is given by
\begin{gather}\label{syl-bin}
\det B = (-1)^{r_1r_2} (s_1-s_2)^{r_1r_2}.
\end{gather}
For example, for $r_1=2$ and $r_2=3$, we have
\begin{gather*}
\det\begin{pmatrix}
1 & 2s_1 & s_1^2 & 0 & 0\\
0 & 1& 2s_1 & s_1^2 & 0\\
0 & 0 &1 & 2s_1 & s_1^2 \\
1 & 3s_2 & 3s_2^2 & s_2^3 & 0\\
0& 1 & 3s_2 & 3s_2^2 & s_2^3
\end{pmatrix}= (s_1-s_2)^6.
\end{gather*}
This explains why the determinant is $0$ when $s_1$ and $s_2$ are both $1$. The matrix of this last example is the Sylvester matrix corresponding to the polynomials
\begin{gather*}
x^2+2s_1x+s_1^2 = (x+s_1)^2
\end{gather*}
and
\begin{gather*}
x^3+3 s_2 x^2 +3s_2^2 x + s_2^3 = (x+s_2)^3.
\end{gather*}

The formula for $\det B$ follows from a well-known result concerning the determinant of Syl\-vester polynomials, given in, for example, Loos~\cite[Theorem~1, p.~177]{Loos1983}. Consider the polynomials
\begin{gather*}
C \prod_{i=1}^{r_1} (x-\gamma_i) \qquad \text{and} \qquad K \prod_{j=1}^{r_2} (x-\beta_j),
\end{gather*}
where $C$, $K$ are constants. Then the determinant of the corresponding Sylvester matrix is given by
\begin{gather*}
C^{r_2}K^{r_1}\prod_{i=1}^{r_1}\prod_{j=1}^{r_2} (\gamma_i-\beta_j).
\end{gather*}

However, this classic idea fails in the evaluation of the determinant of the following ``Syl\-vester\-esque'' matrix, which appears in the context of approximation theory in work of the second author with Feng and Xu~\cite{FKX2017}.

To state their formula, we require the {\it rising factorial} (or {\it Pochhammer symbol}) $\rfac{a}{n}$, which is defined by $\rfac{a}{0}:=1$ and
\begin{gather*}
\rfac{a}{k}:=\prod_{j=0}^{k-1} (a+j), \qquad\text{for $k=1, 2, \dots$}.
 \end{gather*}
Let
\begin{gather*}
m_{ij}(s_1, s_2, r) :=\binom{r}{j-i} \frac{\rfac{s_1+i}{j-i}}
{\rfac{s_1+s_2+ {i+j-1}}{j-i}\rfac{ s_1+s_2 +{r+2i}}{j-i}}.
\end{gather*}
Consider the $ (r_1+r_2) \times (r_1+r_2)$ matrix $M=(m^{\prime}_{ij})$, where
\begin{gather*}
m^{\prime}_{ij} =
\begin{cases}
m_{ij}(s_1, s_2, r_1), & \text{for } 0\leq i \leq r_2-1,\\
(-1)^{j-i-r_2}m_{i-r_2,j}(s_2, s_1, r_2), & \text{for } r_2\leq i \leq r_1+r_2-1 .
\end{cases}
\end{gather*}
This matrix is not quite a Sylvester matrix, but still has a nice determinant evaluation (see \cite[Theorem~4.1]{FKX2017})
\begin{gather}\label{feng-krat-xu}
\det M = (-1)^{r_1r_2} \prod_{j=1}^{r_1} \frac{1} {\rfac{s_1+s_2+ {r_1+r_2+j-2}}{r_2}}.
\end{gather}

It is not obvious, but the formula for $\det{M}$ extends the one for $\det{B}$. To obtain \eqref{syl-bin} from~\eqref{feng-krat-xu}, multiply row $i$ of $\det M$ by
\begin{gather*}
\begin{cases}
\displaystyle \frac{t^i}{(s_1+s_2)^{2i}}, & \text{for $0\leq i \leq r_2-1$,} \\
\displaystyle \frac{t^{i-r_2}}{(s_1+s_2)^{2(i-r_2)}}, & \text{for $r_2\leq i \leq r_1+r_2-1$,}
\end{cases}
\end{gather*}
column $j$ by
\begin{gather*}
\frac{(s_1+s_2)^{2j}}{t^j},\qquad \text{for $0\leq j \leq r_1+r_2-1$},
\end{gather*}
and multiply by the corresponding products on the right hand side of~\eqref{feng-krat-xu}. Next, replace $s_1$ by~$s_1t$, $s_2$ by~$s_2 t$ and take the limit as $t\to \infty$. Equation~\eqref{syl-bin} follows after replacing $s_2$ by $-s_2$ on both sides.

In this paper, we extend the formula for $\det M$ to one with elliptic hypergeometric terms. Our first result (stated in Section~\ref{th1}) is the evaluation of an elliptic determinant, which has two additional parameters even when specialized to the hypergeometric case.

Our proof (in Section~\ref{proof-th1}) is an elliptic extension of the one given in \cite{FKX2017}, and requires a determinant lemma due to Warnaar \cite[Lemma~5.3]{SOW2002} and a $C_n$ summation theorem. Warnaar's determinant lemma is an elliptic extension of a~very useful determinant lemma due to the second author \cite[Lemma~5]{Krat1999}. The $C_n$ summation theorem was conjectured by Warnaar~\cite[Corollary~6.2, $x=q$]{SOW2002}, proved by Rosengren~\cite{HR2001} (and in more generality by Rains \cite[Theorem~4.9]{Rains2006} and, independently by Coskun and Gustafson~\cite{CG2006}). A combinatorial proof was given by Schlosser~\cite{MS2007a}.

The summation formula is a $C_n$ extension of Frenkel and Turaev's~\cite{FT1997} $_{10}V_9$ summation formula (see \cite[equation~(11.4.1)]{GR90}). There is a more general transformation formula which is a $C_n$ extension of Frenkel and Turaev's $_{12}V_{11}$ transformation formula, given in \cite[equation~(11.5.1)]{GR90}. Again, this was conjectured by Warnaar~\cite[Conjecture~6.1, $x=q$]{SOW2002}, and proved~-- in more generality, and independently~-- by Rains~\cite[Theorem~4.9]{Rains2006} and by Coskun and Gustafson~\cite{CG2006}. (See also~\cite{HR2018a} for an elementary proof of \cite[Conjecture~6.1, $x=q$]{SOW2002}.)

Naturally, we consider what happens if we use the $C_n$ $_{12}V_{11}$ transformation formula. This leads (in Section~\ref{th2}) to a surprisingly elegant transformation formula between two elliptic determinants. In Section~\ref{2reductionto1}, we show how to recover our determinant evaluation from this transformation formula by using elementary determinant operations.

\section{An elliptic determinant evaluation}\label{th1}
In this section we state our first theorem, an elliptic determinant evaluation. To define the entries of the matrix under consideration, we need the notation for elliptic shifted factorials. For these notations, and background results, Gasper and Rahman \cite[Chapter~11]{GR90} is the standard reference. We recommend Rosengren~\cite{HR2016-lectures} for a friendly introduction to elliptic hypergeometric series.

We follow Gasper and Rahman \cite[(11.6.2)]{GR90} and define products as follows:
\begin{gather}\label{proddef}
\prod_{j=k}^n a_j := \begin{cases}
a_ka_{k+1}\cdots a_n, & {\text{if } n\geq k},\\
1, & {\text{if } n= k-1},\\
(a_{n+1}a_{n+2}\cdots a_{k-1})^{-1}, &{\text{if } n\leq k-2}.
\end{cases}
\end{gather}
The {\em $q$-shifted factorials}, for $k$ any integer, are defined as
\begin{gather*}
\pqrfac{a}{k}{q} :=\prod\limits_{j=0}^{k-1} \big(1-aq^j\big),
\end{gather*}
and for $|q|<1$,
\begin{gather*}
\pqrfac{a}{\infty}{q} := \prod\limits_{j=0}^{\infty} \big(1-aq^j\big).
\end{gather*}
The {\em modified Jacobi theta function} is defined as
\begin{gather*}
 \ellipticthetap{a} := \pqrfac{a}{\infty}{p} \pqrfac{p/a}{\infty}{p},
 \end{gather*}
where $a\neq 0$ and $|p|<1$. The {\em $q, p$-shifted factorials} (or {\em theta shifted factorials}), for $k$ an integer, are defined as
\begin{gather*}
\ellipticqrfac{a}{k} := \prod\limits_{j=0}^{k-1} \ellipticthetap{aq^j}.
\end{gather*}
The parameters $p$ and $q$ are called the {\em nome} and {\em base}, respectively. When $p=0$, the modified theta function $\ellipticthetap{a}$ reduces to $(1-a)$; and thus $\ellipticqrfac{a}{k}$ reduces to $ \pqrfac{a}{k}{q}$.

Further, we use the short-hand notations
\begin{gather*}
\ellipticthetap{a_1, a_2, \dots, a_r} := \ellipticthetap{a_1} \ellipticthetap{a_2}\cdots \ellipticthetap{a_r},\\
\ellipticqrfac{a_1, a_2,\dots, a_r}{k} := \ellipticqrfac{a_1}{k} \ellipticqrfac{a_2}{k}\cdots\ellipticqrfac{a_r}{k},\\
\pqrfac{a_1, a_2,\dots, a_r}{k}q := \pqrfac{a_1}{k}q \pqrfac{a_2}{k}q\cdots\pqrfac{a_r}{k}q.
\end{gather*}

Next, we define the elliptic matrix under consideration. Let $s_1$, $s_2$, $t_1$ and $t_2$ be arbitrary complex numbers, $r_1$, $r_2$, $i$, and $j$ be non-negative integers, and let
\begin{gather}
f_{ij}=f_{ij}(s_1, s_2, t_1, t_2, r_1, r_2) :=
q^{{\binom {j- i} 2}+r_2(j-i)} \frac{\ellipticqrfac{q}{r_1}}{\ellipticqrfac{q}{r_1-j+i}}\nonumber\\
\hphantom{f_{ij}=f_{ij}(s_1, s_2, t_1, t_2, r_1, r_2) :=}{} \times \frac{\ellipticqrfac{s_1q^i, t_1q^i, t_2q^i, s_1s_2^2 q^{r_1-r_2+i}/t_1t_2}{j-i}}
{\ellipticqrfac{q, s_1s_2q^{i+j-1}, s_1s_2 q^{r_1+2i}}{j-i}}.\label{f}
\end{gather}
We consider the $ (r_1+r_2) \times (r_1+r_2)$ matrix $F=(f^{\prime}_{ij})$, where
\begin{gather*}
f^{\prime}_{ij} =
\begin{cases}
f_{ij}(s_1, s_2, t_1, t_2, r_1, r_2), & \text{for } 0\leq i \leq r_2-1,\\
f_{i-r_2,j}(s_2, s_1, s_1s_2/t_1, s_1s_2/t_2, r_2, r_1), & \text{for
} r_2\leq i \leq r_1+r_2-1 .
\end{cases}
\end{gather*}
The form of $F$ is similar to that of Sylvester matrices. Observe that, in view of \eqref{proddef}, we have
\begin{gather*}
\frac{1}{\qrfac{q}{-r}} =0, \qquad \text{for } r = 1, 2, 3, \dots.
\end{gather*}
Thus, for the first $r_2$ rows, $m_{ij}=0$ if $j<i$, and if $j>r_1+i$. Further, the first non-zero entry is $1$ (when $i=j$). Similar remarks apply for the next $r_1$ rows of $F$.

\begin{Theorem}\label{GB-Krat-det2} Let $F$ be the matrix defined above. Then
\begin{gather}
\det{F} = (-1)^{r_1r_2} \left( \frac{t_1t_2}{s_2}\right)^{r_1r_2} q^{\frac{1}{2} r_1r_2(r_1+4r_2-3)} \nonumber\\
\hphantom{\det{F} =}{} \times \prod_{j=1}^{r_1} \frac{\ellipticqrfac{s_2q^{-r_2+j}/t_1,
 s_2q^{-r_2+j}/t_2,s_1s_2q^{-r_2+j}/t_1t_2}{r_2}} {\ellipticqrfac{s_1s_2q^{r_1+r_2+j-2}}{r_2}}.\label{GB-Krat-det2-eqn}
\end{gather}
\end{Theorem}

\begin{rem} Theorem~\ref{GB-Krat-det2} can also be stated as
\begin{gather}
\det{F} = s_1^{r_1r_2}q^{r_1r_2(r_1+r_2-1)+r_1{\binom {r_2} 2}} \nonumber\\
\hphantom{\det{F} =}{} \times\prod_{j=1}^{r_1} \frac{\ellipticqrfac{s_2q^{-r_2+j}/t_1, s_2q^{-r_2+j}/t_2,
 t_1t_2q^{-r_1+j}/s_1s_2}{r_2}} {\ellipticqrfac{s_1s_2q^{r_1+r_2+j-2}}{r_2}}.\label{det2-form2}
\end{gather}
\end{rem}

Next we take a special case to obtain a $q$-analogue of \eqref{feng-krat-xu}. Let
\begin{gather}\label{u}
u_{ij}=u_{ij}(s_1, s_2, r_1, r_2) := q^{{\binom {j- i} 2}}\frac{\pqrfac{q}{r_1}q \pqrfac{s_1q^i}{j-i}q}{\pqrfac{q}{r_1-j+i}q \pqrfac{q, s_1s_2q^{i+j-1}, s_1s_2 q^{r_1+2i}}{j-i}q}.
\end{gather}
We consider the $ (r_1+r_2) \times (r_1+r_2)$ matrix $U=(u^{\prime}_{ij})$, where
\begin{gather*}
u^{\prime}_{ij} = \begin{cases}
q^{\binom{j}2-\binom{i}2} u_{ij}(s_1, s_2, r_1, r_2), & \text{for } 0\leq i \leq r_2-1,\\
\left({-1}/{s_2}\right)^{j-i+r_2} u_{i-r_2,j}(s_2, s_1, r_2, r_1), & \text{for
} r_2\leq i \leq r_1+r_2-1 .
\end{cases}
\end{gather*}
\begin{Corollary}\label{cor:qfengxu} Let $U$ be the matrix defined above. Then
\begin{gather}\label{qfengxu}
\det{U} = \frac{ (-1)^{r_1r_2} }{s_2^{r_1r_2}} q^{r_1\binom{r_2}{2}}\prod_{j=1}^{r_1} \frac{1}{\pqrfac{s_1s_2q^{r_1+r_2+j-2}}{r_2}q}.
\end{gather}
\end{Corollary}
\begin{rems}
(1) This determinant evaluation is a $q$-analogue of \eqref{feng-krat-xu}. To see this, we replace $s_1$ and $s_2$ by $q^{s_1}$ and $q^{s_2}$ (respectively), multiply both sides of \eqref{qfengxu} by $(1-q)^{r_1r_2}$ and take the limits as $q\to 1$.

To take entry-wise limits in the determinant, we need to multiply each entry by an appropriate power of~$(1-q)$. Thus we multiply each entry of column $j$ by $(1-q)^j$, and divide the first $r_2$ rows by $(1-q)^i$, and the last $r_1$ rows by $(1-q)^{i-r_2}$. We compensate by multiplying the resulting determinant by
\begin{gather*}
\prod_{j=0}^{r_1+r_2-1} \frac{1}{(1-q)^j}\prod_{i=0}^{r_2-1} {(1-q)^i}\prod_{i=r_2}^{r_1+r_2-1} {(1-q)^{i-r_2}} = \frac{1}{(1-q)^{r_1r_2}}.
\end{gather*}
This explains why we need to multiply by $(1-q)^{r_1r_2}$ on the left hand side. The limit on the right hand side too requires this additional power of $(1-q)$.

(2) Corollary~\ref{cor:qfengxu} being a $q$-analogue of \eqref{feng-krat-xu}, it may be the starting point for finding a~$q$-analogue for the best approximation result in \cite{FKX2017}.
\end{rems}

\begin{proof}We take $p=0$ in \eqref{GB-Krat-det2-eqn}, and after dividing both sides by $(t_1t_2)^{r_1r_2}$, take the limits as $t_1\to\infty$ and $t_2\to\infty$. From the right hand side, we obtain
\begin{gather*}
\frac{ (-1)^{r_1r_2} }{s_2^{r_1r_2}}q^{\frac{1}{2} r_1r_2(r_1+4r_2-3)}\prod_{j=1}^{r_1} \frac{1}{\pqrfac{s_1s_2q^{r_1+r_2+j-2}}{r_2}q}.
\end{gather*}

Before taking the limits on the left hand side, we use the elementary modifications
\begin{gather*}
\pqrfac{t_1q^i,t_2q^i}{j-i}{q} = \pqrfac{q^{1-j}/t_1,q^{1-j}/t_2}{j-i}{q} (t_1t_2)^{j-i} q^{2\binom{j-i}{2}+2i(j-i)}
\end{gather*}
to reverse some of the products in the first $r_2$ rows of $\det F$, and
\begin{gather*}
\pqrfac{t_1t_2q^{-r_1+i}/s_2}{j-i+r_2}{q}= \pqrfac{q^{1+r_1-r_2-j}s_2/t_1t_2}{j-i+r_2}{q}\\
\hphantom{\pqrfac{t_1t_2q^{-r_1+i}/s_2}{j-i+r_2}{q}=}{} \times (-1)^{j-i+r_2}\left( \frac{t_1t_2}{s_2}\right)^{j-i+r_2} q^{\binom{j-i+r_2}{2}+(-r_1+i)(j-i+r_2)}
\end{gather*}
in the last $r_1$ rows of the determinant. From the $j$th column of the determinant, we take out~$(t_1t_2)^j$. From the first $i$ rows, we can take out~$(t_1t_2)^{-i}$, and from the last $r_1$ rows, we take out $(t_1t_t)^{-(i-r_2)}$. This results in the product $(t_1t_2)^{r_1r_2}$ outside the determinant that cancels the product we divided earlier.

Now we take the limits as $t_1, t_2\to \infty$ term-wise in the determinant, to obtain a determinant whose entries are given by
\begin{gather*}
\begin{cases}
q^{2\binom{j}2-2\binom{i}2 +r_2(j-i)} u_{ij}(s_1, s_2, r_1, r_2), & \text{for } 0\leq i \leq r_2-1,\\
\left({-1}/{s_2}\right)^{j-i+r_2} q^{\binom{j-i+r_2}{2}+i(j-i+r_2)}u_{i-r_2,j}(s_2, s_1, r_2, r_1), & \text{for } r_2\leq i \leq r_1+r_2-1 .
\end{cases}
\end{gather*}

Again, we take out some powers of $q$ from the determinant, and cancel them from the other side to obtain Corollary~\ref{cor:qfengxu}.
\end{proof}

A slightly different limiting case yields another determinant evaluation. Again, let $u_{ij}$ be defined by \eqref{u}. We consider the matrix $V=(v^{\prime}_{ij})$, where
\begin{gather*}
v^{\prime}_{ij} =
\begin{cases}
u_{ij}(s_1, s_2, r_1, r_2), & \text{for } 0\leq i \leq r_2-1,\\
\left({s_1}/{s_2}\right)^{(j-i+r_2)/2} u_{i-r_2,j}(s_2, s_1, r_2, r_1), & \text{for } r_2\leq i \leq r_1+r_2-1 .
\end{cases}
\end{gather*}
\begin{Corollary}\label{cor:qfengxu2} Let $V$ be the matrix defined above. Then
\begin{gather*}
\det{V} = \left(\frac{s_1}{s_2}\right)^{\frac{1}{2}r_1r_2} q^{r_1\binom{r_2}{2}}\prod_{j=1}^{r_1} \frac{\pqrfac{(s_2/s_1)^{1/2}q^{-r_2+j}}{r_2}q}{\pqrfac{(s_1s_2)^{1/2}q^{j-1}, s_1s_2q^{r_1+r_2+j-2}}{r_2}q}.
\end{gather*}
\end{Corollary}

\begin{rem} When we take the hypergeometric case of Corollary~\ref{cor:qfengxu2}, by replacing $s_1$ by~$q^{s_1}$, $s_2$~by~$q^{s_2}$, and taking the limit as $q\to 1$, we obtain the determinant of a matrix very similar to the matrix $M$ considered in \eqref{feng-krat-xu}. The difference is that there are no alternating signs in the last $r_1$ rows of the matrix.
\end{rem}
\begin{proof}We take $p=0$ in \eqref{GB-Krat-det2-eqn}, divide both sides by $t_2^{r_1r_1}$ and take limits as $t_2\to\infty$. Next we replace $t_1$ by $(s_1s_2)^{1/2}$.

In the resulting determinant on the left hand side we multiply the $i$th row by
\begin{gather*}
\begin{cases}
\big( (s_1s_2)^{1/2};q\big)_i, & \text{for $0\leq i \leq r_2-1$,} \\
\big((s_1s_2)^{1/2};q\big)_{i-r_2}, & \text{for $r_2\leq i \leq r_1+r_2-1$,}
\end{cases}
\end{gather*}
and, to compensate, divide the determinant by
\begin{gather*}
\prod\limits_{i=0}^{r_2-1}\big( ( s_1s_2)^{1/2};q\big)_{i}\prod\limits_{i=r_2}^{r_1+r_2-1}\big( ( s_1s_2)^{1/2};q\big)_{i-r_2}.
\end{gather*}
In the resulting determinant, column $j$ has the common factor
\begin{gather*}
\big( (s_1s_2)^{1/2};q\big)_{j}
\end{gather*}
 that we can take out from each column. In addition, we take out negative signs and powers of~$q$ to obtain
\begin{gather*}
(-1)^{r_1r_2} q^{\frac{1}{2}r_1r_2(r_1+3r_2-2)} \frac{\prod\limits_{j=0}^{r_1+r_2-1} \pqrfac{(s_1s_2)^{1/2}}{j}{q}}	{\prod\limits_{i=0}^{r_2-1} \pqrfac{(s_1s_2)^{1/2}}{i}{q}	\prod\limits_{i=0}^{r_1-1}\pqrfac{(s_1s_2)^{1/2}}{i}{q}}\det V.
\end{gather*}

Now comparing with what we obtain on the right hand side after taking the limit $t_2\to\infty$ and $t_1=(s_1s_2)^{1/2}$, we obtain Corollary~\ref{cor:qfengxu2}.
\end{proof}

\section{Proof of Theorem~\ref{GB-Krat-det2}}\label{proof-th1}
In this section we give a proof of Theorem~\ref{GB-Krat-det2}. First we record the required results in a form suitable for our use. From now on, we suppress $q$ and $p$ and denote the theta shifted factorials by $\qrfac{a}{k}$.

We begin with a determinant lemma due to Warnaar \cite[Lemma~5.3]{SOW2002}.
\begin{Proposition}\label{warnaar-determinant}
Let $A_0, \dots, A_{n-1}$, $c$, and $x_0, \dots, x_{n-1}$, be arbitrary complex numbers. If, for $i=0,
1, \dots, n-1$, $P_i$ is analytic in $0<|x|<\infty$, and satisfies the
following two conditions:
\begin{enumerate}\itemsep=0pt
\item[$1)$] $P_i(px)=(c/x^2p)^iP_i(x)$ $($quasi-periodicity$)$, and
\item[$2)$] $P_i(c/x)=P_i(x)$ $($symmetry$)$,
\end{enumerate}
then
\begin{gather*}
\det_{0\leq i,j\leq n-1} \left(P_i(x_j)\prod_{k=i+1}^{n-1}\elliptictheta{A_kx_j}\elliptictheta{cA_k/x_j} \right) \\
\qquad{} = \prod_{0\leq i<j\leq n-1} (cA_j/x_j)\elliptictheta{x_j/x_i}\elliptictheta{x_ix_j/c}\prod_{i=0}^{n-1} P_i(1/A_i).
\end{gather*}
\end{Proposition}
Warnaar's determinant lemma is used to evaluate the following determinant.
\begin{Lemma}\label{det-lemma}
For all positive integers $r_2$, we have
\begin{gather*}
\det_{0\leq i,j\leq r_2-1}
\Bigg( q^{-ik_j} \qrfac{q^{k_j-i+1}, aq^{k_j} }{i} \qrfac{q^{r_1-k_j+i+1}, aq^{r_1+k_j+i+1} }{r_2-i-1}\Bigg)\\
\qquad{} = q^{-\sum\limits_{i=0}^{r_2-1}ik_i}
\prod_{0\leq i<j \leq r_2-1} \elliptictheta{q^{k_j-k_i}, aq^{k_i+k_j}}\prod\limits_{i=0}^{r_2-1} \qrfac{q^{r_1+1}, aq^{r_1+i}}{i}.
\end{gather*}
\end{Lemma}
\begin{proof}
We apply Proposition~\ref{warnaar-determinant} with
\begin{gather*}
P_i(x)=\big( c^{\frac{1}{2}}/{x}\big)^i \qrfac{xq^{1-i}, x/c}{i}.
\end{gather*}
The quasi-periodicity and symmetry properties of $P_i(x)$ are easy to verify using the elementary identities \cite[equation~(11.2.55)]{GR90}
\begin{gather*}
\qrfac{a}{n} =\qrfac{pa}{n}(-a)^n q^{\binom n 2}
\end{gather*}
and \cite[equation~(11.2.48)]{GR90}
\begin{gather*}
\qrfac{a}{n} =\qrfac{q^{1-n}/a}{n}(-a)^n q^{\binom n 2}.
\end{gather*}
The special case $n=r_2$, $x_j=q^{k_j}$, $A_k=q^{-r_1-k}$, $c=1/a$ of Proposition~\ref{warnaar-determinant} implies that
\begin{gather*}
\det_{0\leq i,j\leq r_2-1} \Bigg( q^{2{\binom {i+1} 2}-2{\binom {r_2} 2}-(2r_1-1)(r_2-i-1)}
\left( 1/a \right)^{i/2} \left({1}/{aq}\right)^{r_2-i-1}\\
\qquad\quad{}\times q^{-ik_j} \qrfac{q^{k_j-i+1}, aq^{k_j} }{i} \qrfac{q^{r_1-k_j+i+1}, aq^{r_1+i+k_j+1} }{r_2-i-1}\Bigg)\\
\qquad{} = \prod_{0\leq i<j \leq r_2-1} \left(q^{-r_1-j-k_j}/a\right) \elliptictheta{q^{k_j-k_i}, aq^{k_i+k_j}}\\
\qquad\quad{} \times \prod\limits_{i=0}^{r_2-1} (1/a)^{i/2} \left( {q^{-r_1-i}}\right)^i\qrfac{q^{r_1+1}, aq^{r_1+i}}{i}.
\end{gather*}
To complete the proof of the lemma, we take out the common factors from each row of the determinant, cancel common terms, and simplify.
\end{proof}

Next, we require a $C_n$ extension of Frenkel and Turaev's $_{10}V_8$ summation formula. As explained earlier, this is due to Warnaar and Rosengren, but we prefer to use a formulation due to Schlosser. In Schlosser \cite[Theorem~3.1]{MS2007a}, we take $r=r_2$, use the summation indices $k_0,k_1, \dots, k_{r_2-1}$, make the substitutions $a\mapsto s_1s_2/q$, $b\mapsto s_1$, $c\mapsto t_1$, $d\mapsto t_2$ and take $m=r_1+r_2-1$, to obtain
\begin{gather}
\sum_{0\leq k_0<\dots <k_{r_2-1}\leq r_1+r_2-1} \prod_{j=0}^{r_2-1}\frac{\qrfac{s_1s_2/q, s_1, t_1, t_2, s_1s_2^2q^{r_1-r_2}/t_1t_2,q^{-(r_1+r_2-1)} }{k_j} }{\qrfac{q, s_2, s_1s_2/t_1, s_1s_2/ t_2,
t_1t_2q^{r_2-r_1}/s_2, s_1s_2q^{r_1+r_2-1} }{k_j} }\nonumber\\
\qquad\quad{}\times \prod_{0\leq i<j \leq r_2-1} \elliptictheta{q^{k_j-k_i},s_1s_2q^{k_i+k_j-1}}^2\prod_{j=0}^{r_2-1}\frac{\elliptictheta{s_1s_2q^{2k_j-1} }}{\elliptictheta{s_1s_2/q} }
q^{\sum\limits_{i=0}^{r_2-1} (2r_2-2i-1)k_i} \nonumber\\
\qquad{} = q^{-4{\binom {r_2} 3}} \left( \frac{s_2}{t_1t_2q^2}\right)^{\binom {r_2} 2} \prod_{j=1}^{r_2}\qrfac{q, s_1, t_1, t_2, s_1s_2^2q^{r_1-r_2}/t_1t_2 }{j-1} \qrfac{q, s_1s_2}{r_1+r_2-1} \nonumber\\
\qquad\quad{} \times \prod_{j=1}^{r_2} \frac{\qrfac{ s_2q^{1-j}/t_1,s_2q^{1-j}/t_2, s_1s_2q^{1-j}/t_1t_2}{r_1}}{\qrfac{q, s_2, s_1s_2/t_1, s_1s_2/ t_2,s_2q^{1-2r_2+j}/t_1t_2 }{r_1+r_2-j} }.\label{rosengren}
 \end{gather}
These are all the ingredients required for the proof of Theorem~\ref{GB-Krat-det2}.

\begin{proof}[Proof of Theorem~\ref{GB-Krat-det2}] Let $F^{a_1,a_2,\dots, a_r}_{b_1,a_2,\dots, b_r}$ denote the submatrix of $F$ consisting of rows $a_1, \dots,a_r$ and columns $b_1,\dots, b_r$. By taking the Laplace expansion of $\det F$ with respect to the first $r_2$ rows, we find that
\begin{gather}\label{laplace}
\det F = \sum_{0\leq k_0<\dots <k_{r_2-1}\leq r_1+r_2-1} (-1)^{{\binom {r_2} 2}+\sum\limits_{i=0}^{r_2-1} k_i} \det F^{0,1,\dots, r_2-1}_{k_0,k_1,\dots, k_{r_2-1}} \det F^{r_2,r_2+1,\dots, r_1+r_2-1}_{l_0,l_1,\dots, l_{r_1-1}},
\end{gather}
where $\{ l_0,l_1,\dots, l_{r_1-1}\}$ is the complement of the set $\{k_0,k_1,\dots, k_{r_2-1}\}$ in the set $\{ 0, 1, 2, \dots$, $r_1+r_2-1\}$. Let $D_1$ and $D_2$ denote the two determinants in the sum. Then
\begin{gather*}
D_1=\det F^{0,1,\dots, r_2-1}_{k_0,k_1,\dots, k_{r_2-1}} =\det \left( f_{i,k_j}(s_1, s_2, t_1, t_2, r_1, r_2)\right)
\end{gather*}
and
\begin{gather*}
D_2=\det F^{r_2,r_2+1,\dots, r_1+r_2-1}_{l_0,l_1,\dots, l_{r_1-1}} = \det \left( f_{i,l_j}(s_2, s_1, s_1s_2/t_1, s_1s_2/t_2, r_2, r_1)\right),
\end{gather*}
where the $f_{ij}$'s are as in \eqref{f}. After evaluating these determinants using Lemma~\ref{det-lemma}, the resulting sum can be evaluated using the elliptic $C_n$ summation theorem given in~\eqref{rosengren}. The result follows after performing a large amount of simplification. Here are some more details.

We first evaluate $D_1=\det ( f_{i,k_j})$, where we suppress the dependence on other parameters for the time being. We first take out many factors out of the rows (indexed by~$i$) and columns (indexed by~$j$), with the goal of eliminating all denominators and reducing the number of factors in each entry of the determinant as much as possible. We have
\begin{gather*}
D_1= \det \left( f_{i,k_j}\right) = q^{\sum\limits_{j=0}^{r_2-1} {\binom {k_j} 2}+r_2\sum\limits_{j=0}^{r_2-1}k_j- \frac{1}{3}(2r_2-1){\binom {r_2} 2}}\prod\limits_{i=0}^{r_2-1}
 \frac{\qrfac{q}{r_1}\qrfac{s_1s_2q^{r_1}}{2i} }{\qrfac{s_1,t_1,t_2,s_1s_2^2q^{r_1-r_2}/t_1t_2}{i}} \\
\hphantom{D_1=}{} \times\prod\limits_{j=0}^{r_2-1} \frac{\qrfac{s_1,t_1,t_2,s_1s_2^2q^{r_1-r_2}/t_1t_2}{k_j}}{\qrfac{q, s_1s_2q^{k_j-1}}{k_j} \qrfac{q}{r_1+r_2-k_j-1} \qrfac{s_1s_2q^{r_1}}{r_2+k_j-1}}\\
\hphantom{D_1=}{} \times \det_{0\leq i,j\leq r_2-1}\Bigg( q^{-ik_j} \qrfac{q^{k_j-i+1}, s_1s_2q^{k_j-1} }{i} \qrfac{q^{r_1-k_j+i+1}, s_1s_2q^{r_1+k_j+i} }{r_2-i-1}\Bigg)\\
\hphantom{D_1}{}=q^{\sum\limits_{j=0}^{r_2-1} {\binom {k_j} 2}+r_2\sum\limits_{j=0}^{r_2-1}k_j -\frac{1}{3}(2r_2-1){\binom {r_2} 2}-\sum\limits_{i=0}^{r_2-1}ik_i}
\prod\limits_{i=0}^{r_2-1} \frac{\qrfac{q}{r_1}\qrfac{q^{r_1+1}, s_1s_2q^{r_1+i-1}}{i} \qrfac{s_1s_2q^{r_1}}{2i} }{\qrfac{s_1,t_1,t_2,s_1s_2^2q^{r_1-r_2}/t_1t_2}{i}} \\
\hphantom{D_1=}{} \times\prod\limits_{j=0}^{r_2-1} \frac{\qrfac{s_1,t_1,t_2,s_1s_2^2q^{r_1-r_2}/t_1t_2}{k_j}}{\qrfac{q, s_1s_2q^{k_j-1}}{k_j} \qrfac{q}{r_1+r_2-k_j-1} \qrfac{s_1s_2q^{r_1}}{r_2+k_j-1}}\\
\hphantom{D_1=}{} \times \prod_{0\leq i<j \leq r_2-1} \elliptictheta{q^{k_j-k_i}, s_1s_2q^{k_i+k_j-1}},
\end{gather*}
where the determinant evaluation in the last line is from Lemma~\ref{det-lemma}, with $a\mapsto s_1s_2/q$.

The second determinant
\begin{gather*}
D_2=\det \left( f_{i,l_j}(s_2, s_1, s_1s_2/t_1, s_1s_2/t_2, r_2, r_1)\right)
 \end{gather*}
is obtained from the above by replacing $k_j$ by $l_j$, and simultaneously replacing
\begin{gather*}
(s_1,s_2, t_1, t_2, r_1, r_2) \qquad \text{by} \quad (s_2, s_1, s_1s_2/t_1, s_1s_2/t_2, r_2, r_1).
\end{gather*}
Note that the indices of summation are $k_j$ and the expression we obtain for $D_2$ is in terms of the~$l_j$'s. We use elementary algebraic manipulations to express the products in terms of the~$k_j$'s. For example, we use the \lq\lq inclusion-exclusion" formula
\begin{gather*}
\prod_{0\leq i<j \leq r_1-1} \elliptictheta{q^{l_j-l_i}} = \frac{\prod\limits_{0\leq i<j \leq r_1+r_2-1} \elliptictheta{q^{j-i}}
\prod\limits_{0\leq i<j \leq r_2-1} \elliptictheta{q^{k_j-k_i}} }
{\prod\limits_{j=0}^{r_2-1}\prod\limits_{i=0}^{k_j-1}\elliptictheta{q^{k_j-i}}
\prod\limits_{i=0}^{r_2-1}\prod\limits_{j=k_i+1}^{r_1+r_2-1}\elliptictheta{q^{j-k_i}} } \\
 \hphantom{\prod_{0\leq i<j \leq r_1-1} \elliptictheta{q^{l_j-l_i}}}{} = \frac
{\prod\limits_{i=0}^{r_1+r_2-1}\qrfac{q}{i} \prod\limits_{0\leq i<j \leq r_2-1} \elliptictheta{q^{k_j-k_i}}}
{\prod\limits_{j=0}^{r_2-1} \qrfac{q}{k_j}\prod\limits_{i=0}^{r_2-1}\qrfac{q}{r_1+r_2-k_i-1}};
\end{gather*}
and, using the same idea,
\begin{gather*}
q^{-\sum\limits_{i=0}^{r_1-1}i l_i} = \prod_{0\leq i<j \leq r_1-1} q^{-l_j} = q^{-\sum\limits_{i=0}^{r_2-1}i k_i +\sum\limits_{i=0}^{r_2-1} {\binom {k_i} 2}-(r_1-1){\binom {r_1+r_2} 2}+\sum\limits_{i=1}^{r_1+r_2-1}{\binom i 2} }.
\end{gather*}
After some further algebraic simplification, we obtain
\begin{gather*}
D_2 = q^{\frac{1}{3}(2r_1+2r_2-1) {\binom {r_1+r_2} 2} -r_1\sum\limits_{j=0}^{r_2-1}k_j -\frac{1}{3}(2r_1-1){\binom {r_1} 2}-\sum\limits_{i=0}^{r_2-1}ik_i}
\prod_{0\leq i<j \leq r_2-1} \elliptictheta{q^{k_j-k_i}, s_1s_2q^{k_i+k_j-1}} \\
\hphantom{D_2 =}{} \times\prod\limits_{j=0}^{r_1+r_2-1}\frac{\qrfac{s_2, s_1s_2/t_1, s_1s_2/t_2, t_1t_2q^{r_2-r_1}/s_2}{j}}{\qrfac{s_1s_2q^{r_2}}{r_1+j-1} \qrfac{q}{r_1+r_2-j-1} }\\
\hphantom{D_2 =}{} \times \prod\limits_{j=0}^{r_1-1}
\frac{\qrfac{q}{r_2} \qrfac{q^{r_2+1},s_1s_2q^{r_2+j-1}}{j}\qrfac{s_1s_2q^{r_2}}{2j}}{\qrfac{s_2, s_1s_2/t_1, s_1s_2/t_2, t_1t_2q^{r_2-r_1}/s_2}{j}}\\
\hphantom{D_2 =}{} \times \prod\limits_{j=0}^{r_2-1}\frac{\qrfac{s_1s_2q^{r_2}}{r_1+k_j-1}}{\qrfac{s_2, s_1s_2/t_1, s_1s_2/t_2, t_1t_2q^{r_2-r_1}/s_2}{k_j}\qrfac{s_1s_2q^{2k_j}}{r_1+r_2-k_j-1}}.
\end{gather*}
Next we substitute the above expressions for $D_1$ and $D_2$ in~\eqref{laplace} to obtain, after some algebraic manipulation, the multiple sum
\begin{gather*}
\det F = (-1)^{\binom {r_2} 2} q^{r_1r_2(r_1+r_2-1)} \prod\limits_{j=0}^{r_1+r_2-1}\frac{\qrfac{s_2, s_1s_2/t_1, s_1s_2/t_2, t_1t_2q^{r_2-r_1}/s_2}{j}}{\qrfac{s_1s_2q^{r_2}}{r_1+j-1} \qrfac{q}{r_1+r_2-j-1}}\\
\hphantom{\det F =}{} \times \prod\limits_{j=0}^{r_1-1}\frac{\qrfac{q}{r_2}\qrfac{q^{r_2+1},s_1s_2q^{r_2+j-1}}{j}\qrfac{s_1s_2q^{r_2}}{2j}}{\qrfac{s_2, s_1s_2/t_1, s_1s_2/t_2, t_1t_2q^{r_2-r_1}/s_2}{j}}\\
\hphantom{\det F =}{} \times
\prod\limits_{j=0}^{r_2-1} \frac{\qrfac{q}{r_1}\qrfac{q^{r_1+1},s_1s_2q^{r_1+j-1}}{j} \qrfac{s_1s_2q^{r_1}}{2j}\qrfac{s_1s_2}{r_1} }{\qrfac{s_1, t_1, t_2, s_1s_2^2q^{r_1-r_2}/t_1t_2}{j}
\qrfac{s_1s_2}{r_2} \qrfac{s_1s_2}{r_1+r_2-1} \qrfac{q}{r_1+r_2-1}}\\
\hphantom{\det F =}{} \times \!\!\sum_{0\leq k_0<\dots <k_{r_2-1}\leq r_1+r_2-1} \!\! \Bigg(\prod_{j=0}^{r_2-1}\frac{\qrfac{s_1s_2/q, s_1, t_1, t_2, s_1s_2^2q^{r_1-r_2}/t_1t_2,
q^{-(r_1+r_2-1)} }{k_j} }{\qrfac{q, s_2, s_1s_2/t_1, s_1s_2/ t_2,t_1t_2q^{r_2-r_1}/s_2, s_1s_2q^{r_1+r_2-1} }{k_j} }\\
\hphantom{\det F =}{} \times
\prod_{0\leq i<j \leq r_2-1} \left(\elliptictheta{q^{k_j-k_i},s_1s_2q^{k_i+k_j-1}}\right)^2\prod_{j=0}^{r_2-1}\frac{\elliptictheta{s_1s_2q^{2k_j-1} }}{\elliptictheta{s_1s_2/q} }
q^{\sum\limits_{i=0}^{r_2-1} (2r_2-2i-1)k_i} \Bigg).
\end{gather*}
This sum can be evaluated using \eqref{rosengren}. After replacing the multiple sum by the corresponding products, we again require a large number of algebraic simplifications. For example, we use
\begin{gather*}
\prod_{i=1}^{r_2} \qrfac{s_2q^{1-i}/t_1}{r_1} = \prod_{j=1}^{r_1}\qrfac{s_2q^{-r_2+j}/t_1}{r_2},\\
\prod_{j=0}^{r_1-1} \qrfac{s_1s_2q^{r_2+j-1}}{j} =\prod_{i=0}^{r_1-1} \qrfac{s_1s_2q^{r_2+2i}}{r_1-i-1},\\
\prod_{i=0}^{r_2-1} \frac{\qrfac{s_1s_2}{r_1} \qrfac{s_1s_2q^{r_1}}{2i} \qrfac{s_1s_2q^{r_1+i-1}}{i} }{\qrfac{s_1s_2}{r_1+r_2+i-1} } =1,
\end{gather*}
and other elementary identities. The result then condenses to the right hand side of \eqref{GB-Krat-det2-eqn}.
\end{proof}

\section{A transformation formula for elliptic determinants}\label{th2}
In this section, we derive a transformation formula between two determinants that extends Theorem~\ref{GB-Krat-det2}, by adding two further parameters. It is apparent from our proof in Section~\ref{proof-th1} that the matrices we consider are closely linked with very-well-poised elliptic hypergeometric series (see~\cite{GR90} for the terminology, if required). We make this connection transparent by re-labelling the parameters.

The Sylvesteresque matrices which we study now are defined as follows. Let $a$, $b$, $c$, $d$, $e$, $f$ be arbitrary complex numbers, $r_1$, $r_2$, $i$, and $j$ be non-negative integers, and let
\begin{gather}
g_{ij}(a; b,c,d,e,f; r_1, r_2) := q^{{\binom {j- i} 2}+r_2(j-i)} \frac{\qrfac{q}{r_1}}{\qrfac{q}{r_1-j+i}}\nonumber\\
\hphantom{g_{ij}(a; b,c,d,e,f; r_1, r_2) :=}{} \times
 \frac{\qrfac{bq^i, cq^i, dq^i, eq^i, fq^i, a^3 q^{r_1-r_2+i+3}/bcdef}{j-i}} {\qrfac{q, aq^{i+j}, a q^{r_1+2i+1}}{j-i}}.\label{g}
\end{gather}
We consider the $ (r_1+r_2) \times (r_1+r_2)$ matrices $G=(g^\prime_{ij})$ and $H = (h^\prime_{ij})$, where
\begin{gather*}
g^\prime_{ij} =
\begin{cases}
g_{ij}(a; b,c,d,e,f; r_1, r_2), & \text{for } 0\leq i \leq r_2-1,\\
g_{i-r_2,j}(a; aq/b,aq/c,aq/d,aq/e,aq/f; r_2, r_1), & \text{for } r_2\leq i \leq r_1+r_2-1 ,
\end{cases}
\end{gather*}
and
\begin{gather*}
h^\prime_{ij} = \begin{cases}
g_{ij}(\lambda; \lambda b/a, \lambda c/a, \lambda d/a, e, f; r_1, r_2), & \text{for } 0\leq i \leq r_2-1,\\
g_{i-r_2,j}(\lambda; aq/b,aq/c,aq/d,\lambda q/e,\lambda q/f; r_2, r_1), & \text{for } r_2\leq i \leq r_1+r_2-1 .
\end{cases}
\end{gather*}

\begin{Theorem}\label{GB-Krat-det2-transformation} With $G$ and $H$ as defined above, and $\lambda = a^2q^{2-r_2}/bcd$, we have
\begin{gather*}
\det{G} = \left( \frac{a}{\lambda}\right)^{r_1r_2} \prod_{j=1}^{r_1} \frac{\qrfac{\lambda q^{r_1+r_2+j-1}}{r_2}}{\qrfac{aq^{r_1+r_2+j-1}}{r_2}} \det{H}.
\end{gather*}
\end{Theorem}

\begin{rem}We reiterate that we have suppressed the $q$, $p$ in our notation.
\end{rem}

\begin{proof}[Sketch of proof] The first few steps in the proof are similar to those of Theorem~\ref{GB-Krat-det2}. We begin with the Laplace expansion of $\det G$ with respect to the first $r_2$ rows to obtain an expression of the form
\begin{gather*}
\det G = \sum_{0\leq k_0<\dots <k_{r_2-1}\leq r_1+r_2-1} (-1)^{{\binom {r_2} 2}+\sum\limits_{i=0}^{r_2-1} k_i} D_1 \cdot D_2,
\end{gather*}
where $D_1$ and $D_2$ are the determinants
 \begin{gather*}
 D_1=\det \left( g_{i,k_j}(a; b,c,d,e,f; r_1, r_2)\right)
 \end{gather*}
and
\begin{gather*}
D_2=\det \left( g_{i,l_j}(a; aq/b,aq/c,aq/d,aq/e,aq/f; r_2, r_1)\right).
\end{gather*}
Here the $g_{ij}$'s are as in \eqref{g} and the $l_j$'s have the same meaning as in the proof of Theorem~\ref{GB-Krat-det2}.

Again, both determinants $D_1$ and $D_2$ can be evaluated using Lemma~\ref{det-lemma}, after taking out common factors from rows and columns of each determinant; and again, the expression for $D_2$ is in terms of the $l_j$'s and we have to write it in terms of the indices of summation $k_j$. After some algebraic manipulations, just as done earlier, we arrive at a $C_n$ multiple sum. This multiple sum can be transformed into a~multiple of another multiple sum using the $C_n$ transformation formula \eqref{RCG} below.

Next, we perform these steps in reverse. Using Lemma~\ref{det-lemma}, and some algebraic simplification, we write the summand of the resulting sum in the form
 \begin{gather*}
(\star\star\star) \sum_{0\leq k_0<\dots <k_{r_2-1}\leq r_1+r_2-1} (-1)^{{\binom {r_2} 2}+\sum\limits_{i=0}^{r_2-1} k_i} D_3 \cdot D_4,
\end{gather*}
where $(\star\star\star)$ are some explicit products, and $D_3$ and $D_4$ are the determinants
\begin{gather*}
D_3=\det \left( g_{i,k_j}(\lambda; \lambda b/a, \lambda c/a, \lambda d/a, e, f; r_1, r_2)\right)
 \end{gather*}
and
\begin{gather*}
D_4=\det \left( g_{i,l_j}(\lambda; aq/b,aq/c,aq/d,\lambda q/e,\lambda q/f; r_2, r_1)\right).
 \end{gather*}
This is the Laplace expansion of $\det{H}$ with respect to the first $r_2$ rows, multiplied by some products.

These products simplify considerably and we obtain the right hand side of Theorem~\ref{GB-Krat-det2-transformation}.
\end{proof}

The $C_n$ transformation formula we require for our proof is a formula due to Warnaar, Rains, and Coskun and Gustafson, as explained earlier. We use a formulation presented in the second author's paper with Schlosser~\cite[Theorem~2]{KS2014}, where we use the summation indices $k_0,k_1, \dots, k_{r_2-1}$, take $r=r_2$ and $m=r_1+r_2-1$, in order to write it in the form
\begin{gather}
 \sum_{0\leq k_0<\dots <k_{r_2-1}\leq r_1+r_2-1} \Bigg( q^{\sum\limits_{i=0}^{r_2-1} (2r_2-2i-1)k_i}\prod_{0\leq i<j \leq r_2-1}\!\! \elliptictheta{q^{k_j-k_i}, aq^{k_i+k_j}}^2 \nonumber\\
\qquad\quad{} \times \prod_{j=0}^{r_2-1}\frac{\elliptictheta{aq^{2k_j} }\qrfac{a, b, c, d, e, f, \lambda aq^{r_1+1}/ef,q^{-(r_1+r_2-1)} }{k_j} }{\elliptictheta{a} \qrfac{q, aq/b, aq/c, aq/d, aq/e, aq/f,
 efq^{-r_1}/\lambda , aq^{r_1+r_2} }{k_j} } \Bigg)\nonumber\\
\qquad {} = \prod_{j=1}^{r_2}\frac{\qrfac{b,c,d, ef/a}{j-1}\qrfac{a q}{r_1+r_2-1} \qrfac{ aq/ef}{r_1} \qrfac{\lambda q/e, \lambda q/f}{r_1+r_2-j} }
{\qrfac{\lambda b/a, \lambda c/a, \lambda d/a, ef/\lambda}{j-1} \qrfac{\lambda q}{r_1+r_2-1} \qrfac{ \lambda q/ef}{r_1} \qrfac{a q/e, a q/f}{r_1+r_2-j} }\nonumber\\
\qquad\quad \times \sum_{0\leq k_0<\dots <k_{r_2-1}\leq r_1+r_2-1} \Bigg( q^{\sum\limits_{i=0}^{r_2-1} (2r_2-2i-1)k_i}\prod_{0\leq i<j \leq r_2-1} \elliptictheta{q^{k_j-k_i}, \lambda
 q^{k_i+k_j}}^2 \nonumber\\
\qquad\quad{} \times \prod_{j=0}^{r_2-1} \frac{\elliptictheta{\lambda q^{2k_j} }\qrfac{\lambda, \lambda b/a,\lambda c/a, \lambda d/a, e, f, \lambda aq^{r_1+1}/ef, q^{-(r_1+r_2-1)} }{k_j} }
{\elliptictheta{\lambda} \qrfac{q, aq/b, aq/c, aq/d, \lambda q/e, \lambda q/f, efq^{-r_1}/a,\lambda q^{r_1+r_2} }{k_j} } \Bigg)
 , \label{RCG}
\end{gather}
where $\lambda = a^2q^{2-r_2}/bcd$.

\section{How Theorem~\ref{GB-Krat-det2-transformation} extends Theorem~\ref{GB-Krat-det2}}\label{2reductionto1}
It is apparent that the transformation formula in Theorem~\ref{GB-Krat-det2-transformation} contains two additional para\-me\-ters than the determinant evaluation in Theorem~\ref{GB-Krat-det2}, and is thus formally an extension of Theorem~\ref{GB-Krat-det2}. By examining the proofs of the two theorems, this fact is confirmed. In this section, we show how to obtain the determinant evaluation in \eqref{det2-form2} directly from the transformation formula in Theorem~\ref{GB-Krat-det2-transformation}, by using elementary determinant operations.

We take $d=aq/c$ in Theorem~\ref{GB-Krat-det2-transformation}. Let $G^\prime=(g^{\prime\prime}_{ij})$ be the resulting matrix. So
\begin{gather*}
g^{\prime\prime}_{ij} =
\begin{cases}
g_{ij}(a; b,c,aq/c,e,f; r_1, r_2), & \text{for } 0\leq i \leq r_2-1,\\
g_{i-r_2,j}(a; aq/b,aq/c,c,aq/e,aq/f; r_2, r_1), & \text{for } r_2\leq i \leq r_1+r_2-1 ,
\end{cases}
\end{gather*}
where the $g_{ij}$'s are as defined in \eqref{g}. Further, let $F^\prime=(f^{\prime\prime}_{ij})$ be the matrix whose entries are given by
\begin{gather*}
f^{\prime\prime}_{ij} =
\begin{cases}
f_{ij}(b, aq/b, e, f, r_1, r_2), & \text{for } 0\leq i \leq r_2-1,\\
f_{i-r_2,j}(aq/b, b, aq/e, aq/f, r_2, r_1), & \text{for } r_2\leq i \leq r_1+r_2-1 ,
\end{cases}
\end{gather*}
where the $f_{ij}$'s are as in \eqref{f}.

We first observe that
\begin{gather}\label{trans-special-lhs}
\det G^\prime =\prod\limits_{i=0}^{r_1-1}\frac{ \qrfac{c, aq/c}{r_2+i}}{ \qrfac{c, aq/c}{i}}\cdot \det F^\prime.
\end{gather}
To see this, we multiply the $i$th row of $\det G^\prime$ by
\begin{gather*}
\begin{cases}
\qrfac{c, aq/c}{i}, & \text{for $0\leq i \leq r_2-1$,} \\
\qrfac{c,aq/c}{i-r_2}, & \text{for $r_2\leq i \leq r_1+r_2-1$,}
\end{cases}
\end{gather*}
and, to compensate, divide $\det G^\prime$ by
\begin{gather*}
\prod\limits_{i=0}^{r_2-1} \qrfac{c, aq/c}{i} \prod\limits_{i=r_2}^{r_1+r_2-1} \qrfac{c, aq/c}{i-r_2}.
\end{gather*}
In this manner we get a determinant equivalent to $\det G^\prime$. But in the resulting determinant, column $j$ has the common factor
\begin{gather*}
\qrfac{c, aq/c }{j}
\end{gather*}
that we can take out from each column, to obtain
\begin{gather*}
\frac{\prod\limits_{j=0}^{r_1+r_2-1} \qrfac{c, aq/c}{j}}{\prod\limits_{i=0}^{r_2-1} \qrfac{c, aq/c}{i} \prod\limits_{i=0}^{r_1-1} \qrfac{c, aq/c}{i}}\det F^\prime.
\end{gather*}
After some cancellation, we obtain \eqref{trans-special-lhs}.

Next, we consider $\det H$ when $d=aq/c$. Again, let $H^\prime = (h^{\prime\prime}_{ij})$, where $h^{\prime\prime}$ is obtained from $h^{\prime}$ by replacing $d$ by $aq/c$. Note that under this substitution, $\lambda = aq^{1-r_2}/b$ and
\begin{gather*}
\lambda b /a = q^{1-r_2}.
\end{gather*}
Thus, in the first $r_2$ rows of $H^\prime$ (i.e., for $0\leq i\leq r_2-1$), $h^{\prime\prime}_{ij}$ contains the factor $\qrfac{q^{1-r_2+i}}{j-i}$. But
\begin{gather*}
\qrfac{q^{1-r_2+i}}{j-i} = \big( q^{-(r_2-i-1)}\big)_{j-i} = 0,\qquad \text{for $j-i> r_2-i-1$},
\end{gather*}
so
\begin{gather*}
h^{\prime\prime}_{ij}= 0, \qquad \text{for } j>r_2-1.
\end{gather*}
This shows that in the first $r_2$ rows of $H^\prime$, the entries in all the columns after the first $r_2$ columns are $0$.

Similarly, for the next $r_1$ rows we find that due to the presence of the factor
\begin{gather*}
\frac{1}{\qrfac{q}{(r_2-j)+(i-r_2)}} = \frac{1}{\qrfac{q}{i-j}}
\end{gather*}
in $h^\prime_{ij}$, we have
\begin{gather*}
h^{\prime\prime}_{ij}= 0, \qquad \text{for } j>i,
\end{gather*}
for $r_2\leq i \leq r_1+r_2-1$. Thus the matrix $H^{\prime}$ is of the form
\begin{gather*}
\left(
\begin{array}{cccc|cccc}
1 & * & * & * & 0 & \hdotsfor 2 & 0\cr
0 & 1 & * & * & 0 & \hdotsfor 2 & 0\cr
\vdots & \vdots & 1 & * & 0 & \hdotsfor 2 & 0\cr
0 & 0 & \dots &1 & 0 & \hdotsfor 2 & 0\cr
\hline
*& * & * & * & * & 0& \dots & 0\cr
0 & * & * & * & * & *& 0 & 0 \cr
\vdots & \vdots & \vdots & \vdots & \vdots & \vdots & * & 0 \cr
0 & 0& \dots & * & \hdotsfor 3 & * \cr
\end{array}
\right).
\end{gather*}
Here the top-left block is an $r_2\times r_2$ upper-triangular matrix, and the bottom right block is an $r_1 \times r_1$ lower-triangular matrix. The determinant of $H^{\prime}$ is given by the product of the diagonal entries. For the first $r_2$ rows, the diagonal entry is $1$. To compute the diagonal entries when $r_2\leq i \leq r_1+r_2-1$, note that
\begin{gather*}
h^{\prime\prime}_{ij} =
 q^{{\binom {j-i+r_2} 2}+r_1(j-i+r_2)} \frac{\qrfac{q}{r_2}}{\qrfac{q}{i-j}}\\
\hphantom{h^{\prime\prime}_{ij} =}{} \times
 \frac{\qrfac{aq^{i-r_2+1}/b, cq^{i-r_2}, aq^{i-r_2+1}/c, \lambda q^{i-r_2+1}/e, \lambda q^{i-r_2+1}/f,
 \lambda b e f q^{-r_1+i-1}/a^2}{j-i+r_2}} {\qrfac{q, \lambda q^{i+j-r_2}, \lambda q^{2i-r_2+1}}{j-i+r_2}}.
\end{gather*}
An expression for $\det H^\prime$ is given by the product of the diagonal entries
\begin{gather*}
\det H^\prime = \prod_{i=r_2}^{r_1+r_2-1} h^{\prime\prime}_{ii}
\end{gather*}
and this gives, after some simplification,
\begin{gather*}
\det H^\prime = q^{r_1^2r_2+r_1{\binom {r_2} 2}} \prod_{i=0}^{r_1-1} 	\frac{\qrfac{c,aq/c}{r_2+i}} {\qrfac{c,aq/c}{i} \qrfac{\lambda q^{r_1+r_2+i}}{r_2} }\\
\hphantom{\det H^\prime =}{} \times
\prod_{i=1}^{r_1} {\qrfac{aq^{-r_2+i+1}/be, aq^{-r_2+i+1}/bf, efq^{-r_1+i-1}/a}{r_2}},
\end{gather*}
where $\lambda = aq^{1-r_2}/b$. From here, we see that, when $d=aq/c$, the right hand side of Theorem~\ref{GB-Krat-det2-transformation} reduces to
\begin{gather}
\left( \frac{a}{\lambda}\right)^{r_1r_2}\prod_{j=1}^{r_1} \frac{\qrfac{\lambda q^{r_1+r_2+j-1}}{r_2}}{\qrfac{aq^{r_1+r_2+j-1}}{r_2}} \det H^\prime =
 b^{r_1r_2} q^{r_1r_2(r_1+r_2-1)+r_1{\binom {r_2} 2}} \nonumber\\
\qquad{} \times \prod_{i=0}^{r_1-1}	\frac{\qrfac{c,aq/c}{r_2+i}}{\qrfac{c,aq/c}{i} }\prod_{i=1}^{r_1}\frac{\qrfac{aq^{-r_2+i+1}/be, aq^{-r_2+i+1}/bf, efq^{-r_1+i-1}/a}{r_2}}
	{\qrfac{a q^{r_1+r_2+i-1}}{r_2} }.\label{trans-special-rhs}
\end{gather}

Now comparing \eqref{trans-special-lhs} and \eqref{trans-special-rhs}, and simultaneously replacing
 \begin{gather*}
 (a, b, e, f) \text{ by } (s_1s_2/q, s_1, t_1, t_2),
 \end{gather*}
 we obtain Theorem~\ref{GB-Krat-det2} in the form \eqref{det2-form2}.

\subsection*{Acknowledgements} We thank Michael Schlosser for helpful discussions. We also thank the referees for many useful suggestions. Research of the first author was supported by a grant of the Austrian Science Fund (FWF), START grant~Y463. Research of the second author was partially supported by the Austrian Science Fund (FWF), grant~F50-N15, in the framework of the Special Research Program ``Algorithmic and Enumerative Combinatorics''.

\pdfbookmark[1]{References}{ref}
\LastPageEnding


\begin{thebibliography}{99}
\footnotesize\itemsep=0pt

\bibitem{CG2006}
Coskun H., Gustafson R.A., Well-poised {M}acdonald functions {$W_\lambda$} and
 {J}ackson coefficients~{$\omega_\lambda$} on {$BC_n$}, in Jack,
 {H}all--{L}ittlewood and {M}acdonald Polynomials, \href{https://doi.org/10.1090/conm/417/07919}{\textit{Contemp. Math.}},
 Vol. 417, Amer. Math. Soc., Providence, RI, 2006, 127--155,
 \href{https://arxiv.org/abs/math.CO/0412153}{math.CO/0412153}.

\bibitem{FKX2017}
Feng H., Krattenthaler C., Xu Y., Best polynomial approximation on the
 triangle, \href{https://arxiv.org/abs/1711.04756}{arXiv:1711.04756}.

\bibitem{FT1997}
Frenkel I.B., Turaev V.G., Elliptic solutions of the {Y}ang--{B}axter equation
 and modular hypergeometric functions, in The {A}rnold--{G}elfand Mathematical
 Seminars, \href{https://doi.org/10.1007/978-1-4612-4122-5_9}{Birkh\"auser Boston}, Boston, MA, 1997, 171--204.

\bibitem{GR90}
Gasper G., Rahman M., Basic hypergeometric series, \href{https://doi.org/10.1017/CBO9780511526251}{\textit{Encyclopedia of
 Mathematics and its Applications}}, Vol.~96, 2nd~ed., Cambridge University Press, Cambridge, 2004.

\bibitem{Krat1999}
Krattenthaler C., Advanced determinant calculus, \textit{S\'em. Lothar.
 Combin.} \textbf{42} (1999), Art.~B42q, 67~pages, \href{https://arxiv.org/abs/math.CO/9902004}{math.CO/9902004}.

\bibitem{KS2014}
Krattenthaler C., Schlosser M.J., The major index generating function of
 standard {Y}oung tableaux of shapes of the form ``staircase minus
 rectangle'', in Ramanujan~125, \href{https://doi.org/10.1090/conm/627/12536}{\textit{Contemp. Math.}}, Vol.~627, Amer. Math.
 Soc., Providence, RI, 2014, 111--122, \href{https://arxiv.org/abs/1402.4538}{arXiv:1402.4538}.

\bibitem{Loos1983}
Loos R., Computing in algebraic extensions, in Computer Algebra, Editors
 B.~Buchberger, G.E. Collins, R.~Loos, R.~Albrecht, \href{https://doi.org/10.1007/978-3-7091-7551-4_12}{Springer}, Vienna, 1983,
 173--187.

\bibitem{Rains2006}
Rains E.M., {$BC_n$}-symmetric {A}belian functions, \href{https://doi.org/10.1215/S0012-7094-06-13513-5}{\textit{Duke Math.~J.}}
 \textbf{135} (2006), 99--180, \href{https://arxiv.org/abs/math.CO/0402113}{math.CO/0402113}.

\bibitem{HR2001}
Rosengren H., A proof of a multivariable elliptic summation formula conjectured
 by {W}arnaar, in {$q$}-Series with Applications to Combinatorics, Number
 Theory, and Physics ({U}rbana, {IL}, 2000), \href{https://doi.org/10.1090/conm/291/04903}{\textit{Contemp. Math.}}, Vol.~291, Amer. Math. Soc., Providence, RI, 2001, 193--202,
 \href{https://arxiv.org/abs/math.CA/0101073}{math.CA/0101073}.

\bibitem{HR2016-lectures}
Rosengren H., Elliptic hypergeometric functions, in Lectures at OPSF-S6,
 College Park, Maryland, July 2016, \href{https://arxiv.org/abs/1608.06161}{arXiv:1608.06161}.

\bibitem{HR2018a}
Rosengren H., Determinantal elliptic {S}elberg integrals, \href{https://arxiv.org/abs/1803.05186}{arXiv:1803.05186}.

\bibitem{MS2007a}
Schlosser M., Elliptic enumeration of nonintersecting lattice paths,
 \href{https://doi.org/10.1016/j.jcta.2006.07.002}{\textit{J.~Combin. Theory Ser.~A}} \textbf{114} (2007), 505--521,
 \href{https://arxiv.org/abs/math.CO/0602260}{math.CO/0602260}.

\bibitem{SOW2002}
Warnaar S.O., Summation and transformation formulas for elliptic hypergeometric
 series, \href{https://doi.org/10.1007/s00365-002-0501-6}{\textit{Constr. Approx.}} \textbf{18} (2002), 479--502,
 \href{https://arxiv.org/abs/math.QA/0001006}{math.QA/0001006}.

\end{thebibliography}
\end{document}